\documentclass[11pt]{article}

\usepackage{latexsym}
\usepackage{epsfig}
\usepackage{amsmath}
\usepackage{amsfonts}
\usepackage{amssymb}
\usepackage{wrapfig}
\usepackage{float}
\usepackage{bm}

\newtheorem{theorem}{Theorem}
\newtheorem{definition}{Definition}
\newtheorem{lemma}{Lemma}

\newtheorem{corollary}{Corollary}
\newtheorem{remark}{Remark}
\newtheorem{example}{Example}

\newenvironment{proof}[1][Proof]{\begin{trivlist}
\item[\hskip \labelsep {\bfseries #1}]}{\end{trivlist}}

\title{Mixing Time of Glauber Dynamics With Parallel Updates and Heterogeneous Fugacities}

\author{Mathieu Leconte, Jian Ni, and R. Srikant\thanks{The authors are with the Coordinated Science Laboratory and Department of Electrical and Computer Engineering, University of
Illinois at Urbana-Champaign, Urbana, IL 61801, USA (leconte2@illinois.edu, jianni@illinois.edu,
rsrikant@illinois.edu).}}

\begin{document}

\date{}

\maketitle

\begin{abstract}
Glauber dynamics is a powerful tool to generate randomized, approximate solutions to combinatorially
difficult problems. Applications include Markov Chain Monte Carlo (MCMC) simulation and distributed
scheduling for wireless networks. In this paper, we derive bounds on the mixing time of a generalization of
Glauber dynamics where multiple vertices are allowed to update their states in parallel and the fugacity of
each vertex can be different. The results can be used to obtain various conditions on the system parameters
such as fugacities, vertex degrees and update probabilities, under which the mixing time grows polynomially
in the number of vertices.
\end{abstract}

\section{Introduction}

Consider a graph $G=(V,E)$, where $V$ is the set of vertices and $ E$ is the set of edges. Suppose $| V|=n$.
For each vertex $v\in  V$, we use $\mathcal{N}_v=\{w\in V: (v,w)\in  E\}$ to denote the set of neighbors of
$v$ in the graph. An \emph{independent set} of $ G$ is a subset of the vertices where no two vertices are
neighbors of each other. Let $\mathcal{I}$ be the set of all independent sets of $ G$.

A \emph{configuration} of the vertices in $G$ is a vector of the form $(\sigma_v)_{v\in V}$, with
$\sigma_v\in\Lambda=\{0, 1\}$ for all $v\in V$. For a vertex $v$ and a configuration $\sigma\in\Lambda^n$, we
say $v\in\sigma$ if $\sigma_v=1$. A configuration $\sigma$ on $ G$ is \emph{feasible} if the set $\{v\in
V:\sigma_v=1\}$ is an independent set of $ G$, i.e., if
\begin{eqnarray}
\sigma_v+\sigma_w & \leq & 1, \mbox{ for all } (v,w)\in  E.
\end{eqnarray} Let
$\Omega\subseteq\Lambda^n$ be the set of all feasible configurations on $ G$.

We associate each vertex $v \in  V$ with a parameter $\lambda$. We are interested in the following
\emph{product-form distribution} over the feasible configurations (independent sets) of the graph:
\begin{eqnarray}
\pi(\sigma) & = & \frac{1}{Z}\prod_{v \in \sigma}\lambda, \label{eq:productform} \\
Z & = & \sum_{\sigma \in \Omega} \prod_{v \in \sigma}\lambda.
\end{eqnarray}
Note that this corresponds to the so-called \emph{hard-core gas model} studied in statistical physics, where
$\lambda$ is called the \emph{fugacity} (e.g., \cite{mar99,vig01}).

\emph{Glauber dynamics} is a Markov chain which generates the stationary distribution in
(\ref{eq:productform}). It has many applications in statistical physics and computer science (e.g., hard-core
gas model, graph coloring, approximate counting, combinatorial optimization \cite{dyegre98,dyegre00,mar99}).
Under (single-site) Glauber dynamics, in each time slot one vertex is selected uniformly at random, and only
that vertex can change its state while other vertices keep their states unchanged. Let $\sigma(t)$ be
the state of the Markov chain in time slot $t$.\\

\noindent\hrulefill
\\
\textbf{Single-Site Glauber Dynamics (in Time Slot $t$)}

\noindent\hrulefill
\begin{itemize}

\item[1.] Choose a vertex $v\in  V$ uniformly at random.

\item[2.] For vertex $v$:\\
\phantom{aa} \textbf{If} $\sum_{w\in \mathcal{N}_v} \sigma_w(t-1)=0$  \\
\phantom{aaaa} (a) $\sigma_v(t)=1$ with probability $p=\frac{\lambda}{1+\lambda}.$ \\
\phantom{aaaa} (b) $\sigma_v(t)=0$ with probability $\bar{p}=\frac{1}{1+\lambda}.$ \\
\phantom{aa} \textbf{Else}\\
\phantom{aaaa} (c) $\sigma_v(t)=0$.

\item[3.] For any vertex $w\in  V\setminus\{v\}$: \\
\phantom{aaaa} (d) $\sigma_w(t)=\sigma_w(t-1)$.
\end{itemize}
\noindent\hrulefill\\

It is not hard to verify that the Glauber dynamics Markov chain is \emph{reversible} and has the product-form
distribution in (\ref{eq:productform}). In most applications, the performance of the Glauber dynamics is
determined by how fast the Markov chain converges to the stationary distribution. The Glauber dynamics is
said to have the \emph{fast (rapid) mixing property} if the mixing time is polynomial in the size of the
graph (the number of vertices $n$).  In \cite{vig01} it was shown that single-site Glauber dynamics has a
mixing time of $O(n \log n)$ when $\lambda < \frac{2}{\Delta-2},$ where $\Delta$ is the maximum vertex degree
in the graph.

Recently, Glauber dynamics has been applied to design distributed throughput-optimal scheduling algorithms
for wireless networks (e.g., \cite{jiawal08,nisri09,rajshashi09}). In the wireless network setting, the graph
$G=(V,E)$ corresponds to the \emph{interference graph} of the wireless network, where the vertices in $V$
represent \emph{links} (transmitter-receiver pairs) in the network, and there is an edge between two vertices
in $G$ if the corresponding wireless network links interfere with each other. A feasible schedule of the
network is a set of links which do not interfere with each other, which corresponds to an independent set in
the interference graph $G$. To achieve maximum throughput, the fugacities need to be chosen as appropriate
functions of the queue lengths of the links, which are normally different from link to link. This motivates
the study of Glauber dynamics with \emph{heterogenous fugacities}.

In addition, in wireless networks, potentially multiple network links (vertices in $G$) can update their
states in a single time slot, and we would expect that the mixing time of the Glauber dynamics Markov chain
will be reduced with such parallel updates. However, the Markov chain may not even have the product-form
distribution (which is a key property for establishing throughput-optimality in
\cite{jiawal08,nisri09,rajshashi09}) if we let an
arbitrary set of vertices update their states. The following two questions then arise:\\
(1) \emph{How to select the vertices in each time slot to update their states such that the
product-form distribution is maintained?} \\
(2) \emph{What is the mixing time of such a Glauber dynamics with parallel updates?}

The first question has been addressed in \cite{nisri09} and the second question will be addressed in this
paper. The rest of the paper is organized as follows. In Section~\ref{sec:parallelGD} we introduce a
generalization of Glauber dynamics with parallel updates and heterogenous fugacities. In
Section~\ref{sec:prelim} we provide some technical background on the mixing time of Markov chains. In
Sections~\ref{sec:mixingPGD} and \ref{sec:mixingHF} we derive bounds on the mixing time of the Glauber
dynamics with parallel updates and heterogenous fugacities. The paper is concluded in
Section~\ref{sec:conclusion}.

\section{Glauber dynamics with parallel updates}  \label{sec:parallelGD}

In \cite{nisri09} we have introduced a generalization of Glauber dynamics where multiple vertices (wireless
network links) are allowed to update their states in a single time slot, under which the Markov chain is
reversible and retains the product-form distribution. The key idea is that in every time slot, we select an
independent set of vertices $\mathbf m \in \mathcal I$ to update their states according to a distributed
randomized procedure, i.e., we select $\mathbf m \in \mathcal I$ with probability $q_{\mathbf m}$, where
$\sum_{\mathbf m \in \mathcal I} q_{\mathbf m}=1.$ We call $\mathbf m$ the \emph{update set} (or
\emph{decision schedule} in \cite{nisri09}). The parallel Glauber dynamics is formally described as
follows.\\

\noindent\hrulefill
\\
\textbf{Parallel Glauber Dynamics (in Time Slot $t$)}

\noindent\hrulefill
\begin{itemize}
\item[1.] Randomly choose an update set $\mathbf m\in \mathcal I$ with probability $q_{\mathbf m}$.

\item[2.] For all vertex $v\in \mathbf m$:\\
\phantom{aa} \textbf{If} $\sum_{w\in \mathcal{N}_v} \sigma_w(t-1)=0$  \\
\phantom{aaaa} (a) $\sigma_v(t)=1$ with probability $p_v=\frac{\lambda_v}{1+\lambda_v}.$ \\
\phantom{aaaa} (b) $\sigma_v(t)=0$ with probability $\bar{p}_v=\frac{1}{1+\lambda_v}.$ \\
\phantom{aa} \textbf{Else}\\
\phantom{aaaa} (c) $\sigma_v(t)=0$.

\item[] For all vertex $w \notin \mathbf m:$\\
\phantom{aaaa} (d) $\sigma_w(t)=\sigma_w(t-1)$.

\end{itemize}
\noindent\hrulefill \\

The following results on the parallel Glauber dynamics have been established in \cite{nisri09}.

\begin{lemma}  \label{lemma:feasible}
Let $\mathbf m(t)$ be the update set selected in time slot $t$. If $\sigma(t-1) \in \Omega$ and $\mathbf m(t)
\in \mathcal I$, then $\sigma(t) \in \Omega$.
\end{lemma}

Because $\sigma(t)$ only depends on the previous state $\sigma(t-1)$ and some randomly selected update set
$\mathbf m(t)$, $\sigma(t)$ evolves as a discrete-time Markov chain (DTMC). Next we will derive the
transition probabilities between the states.

\begin{lemma} \label{lemma:transition}
A state $\sigma \in \Omega$ can make a transition to a state $\eta \in \Omega$ if and only if $\sigma \cup
\eta \in \Omega$ and there exists an update set $\mathbf m\in \mathcal I$ with $q_{\mathbf m}>0$ such that
$$\sigma \bigtriangleup \eta=(\sigma \setminus \eta)\cup(\eta
\setminus \sigma) \subseteq \mathbf m,$$ and in this case the transition probability from $\sigma$ to $\eta$
is given by:
\begin{eqnarray} \label{eq:transition}
P(\sigma, \eta)  & = & \sum_{\mathbf m\in \mathcal I : \sigma \bigtriangleup \eta \subseteq \mathbf
m}q_{\mathbf m} \Big(\prod_{v\in \mathbf \sigma \setminus \eta}\bar{p}_v\Big) \Big(\prod_{v\in \eta\setminus
\sigma}p_v\Big) \Big(\prod_{v\in \mathbf m \cap (\sigma \cap \eta)}p_v\Big) \Big(\prod_{v\in \mathbf m
\setminus (\sigma\cup \eta) \setminus \mathcal{N}_{\sigma \cup \eta}}\bar{p}_v\Big). \nonumber \\
\end{eqnarray}
\end{lemma}

Under the Glauber dynamics with parallel updates, let
$$q_v = \sum_{\mathbf m\ni v}q_{\mathbf m}$$
be the probability of updating vertex $v$ in a time slot.

\begin{theorem} \label{theorem:productform}
A necessary and sufficient condition for the Markov chain of the parallel Glauber dynamics to be irreducible
and aperiodic is $\cup_{\mathbf m\in \mathcal I: q_{\mathbf m}>0}\mathbf m = V$, or equivalently, $q_v>0$ for
all $v\in V$, and in this case the Markov chain is reversible and has the following product-form stationary
distribution:
\begin{eqnarray}
\pi(\sigma) & = & \frac{1}{Z}\prod_{v \in \sigma}\lambda_v, \label{eq:productformPGD} \\
Z & = & \sum_{\sigma \in \Omega} \prod_{v \in \sigma}\lambda_v.
\end{eqnarray}
\end{theorem}

\begin{remark}
The single-site Glauber dynamics can be viewed as a special case of the parallel Glauber dynamics in which
$q_{\mathbf m}>0$ if and only if the update set $\mathbf m$ always consists of only one vertex.
\end{remark}

In this paper we will show that the parallel Glauber dynamics has a very fast mixing time $O(\log n)$ under
certain conditions for bounded-degree graphs. On the other hand, it was shown in \cite{haysin05} that the
single-site Glauber dynamics has a mixing time at least $\Omega(n \log n)$ for bounded-degree graphs.

\section{Mixing Time of Markov Chains}  \label{sec:prelim}

Consider a finite-state, irreducible, aperiodic Markov chain $(P,\Omega, \pi)$ where $P$ denotes the
transition matrix, $\Omega$ denotes the state space, and $\pi$ denotes the unique stationary distribution of
the Markov chain.

\begin{definition}
The \emph{variation distance} between two distributions $\mu, \nu$ on $\Omega$ is defined as
\begin{eqnarray}
||\mu-\nu||_{var} = \frac{1}{2}\sum_{x\in \Omega}|\mu(x)-\nu(x)|.
\end{eqnarray}
\end{definition}

\begin{definition}
The \emph{mixing time} $T_{mix}(\epsilon)$ for $\epsilon>0$ of the Markov chain is defined as the time
required for the Markov chain to get close to the stationary distribution. More precisely,
\begin{eqnarray}
T_{mix}(\epsilon) = \max_{x\in \Omega} \inf\Big\{t: ||P^t(x,\cdot)-\pi||_{var}) \leq \epsilon\Big\}.
\end{eqnarray}
\end{definition}

\begin{definition}
A \emph{coupling} of the Markov chain is a stochastic process $(X(t), Y(t))$ on $\Omega \times \Omega$ such
that $\{X(t)\}$ and $\{Y(t)\}$ marginally are copies the original Markov chain, and if $X(t)=Y(t)$, then
$X(t+1)=Y(t+1)$.
\end{definition}

Let $\Phi$ be a distance function (\emph{metric}) defined on $\Omega \times \Omega$, which satisfies that for
any $x,y,z \in \Omega$:
\begin{itemize}
\item[(1)] $\Phi(x,y)\geq 0$, with equality if and only if $x=y$.

\item[(2)] $\Phi(x,y) = \Phi(y,x).$

\item[(3)] $\Phi(x,z) \leq \Phi(x,y)+\Phi(y,z).$
\end{itemize}

Let $$D_{min} = \min_{x,y\in \Omega, x\neq y}\Phi(x,y), \mbox{ }D_{max} = \max_{x,y\in
\Omega}\Phi(x,y),\mbox{ } D=\frac{D_{max}}{D_{min}}.$$

The following result can be used to obtain an upper bound for the mixing time of the Markov chain (e.g.,
\cite{dyegre98}).

\begin{theorem} \label{theorem:coupling}
Suppose there exist a constant $\beta<1$ and a coupling $(X(t),Y(t))$ of the Markov chain such that, for all
$x,y\in \Omega$,
\begin{eqnarray} \label{eq:betacoupling}
E[\Phi(X(t+1),Y(t+1))|X(t)=x,Y(t)=y] \leq \beta \Phi(x,y).
\end{eqnarray}
Then the mixing time of the Markov chain is bounded by:
\begin{eqnarray} \label{eq:mixingbound}
T_{mix}(\epsilon)  \leq \frac{\log (D\epsilon^{-1})}{1-\beta}. 
\end{eqnarray}
\end{theorem}

In general, determining $\beta$ is hard since one needs to check the contraction condition
(\ref{eq:betacoupling}) for all pairs of configurations. In \cite{bubdye97} the so-called \emph{path
coupling} method was introduced by Bubley and Dyer to simplify the calculation. Under path coupling, we only
need to check the contraction condition for certain pairs of configurations. The path coupling method is
described in the following theorem.

\begin{theorem}  \label{theorem:pathcoupling}
Let $S \subseteq \Omega \times \Omega$ and suppose for all $X, Y \in \Omega \times \Omega$, there exists a
path $X=Z_0, Z_1, \ldots, Z_r=Y$ between $X$ and $Y$ such that $(Z_l,Z_{l+1})\in S$ for $0\leq l <r$ and
$$\Phi(X,Y) = \sum_{l=0}^{r-1}\Phi(Z_l, Z_{l+1}).$$
Suppose there exist a constant $\beta<1$ and a coupling $(X(t),Y(t))$ of the Markov chain such that for any
$(x,y) \in S$,
$$E[\Phi(X(t+1),Y(t+1))|X(t)=x,Y(t)=y] \leq \beta \Phi(x,y).$$
Then the mixing time of the Markov chain is bounded by:
\begin{eqnarray} \label{eq:pathmixingbound}
T_{mix}(\epsilon) \leq  \frac{\log (D\epsilon^{-1})}{1-\beta}.
\end{eqnarray}
\end{theorem}

Note that the key simplification in the path coupling theorem (Theorem~\ref{theorem:pathcoupling}), compared
to the coupling theorem (Theorem~\ref{theorem:coupling}), is that the contraction condition
(\ref{eq:betacoupling}) needs to hold only for $(x,y) \in S$, instead of $(x,y) \in \Omega\times\Omega$.

\section{Mixing Time of Glauber Dynamics With Parallel Updates}  \label{sec:mixingPGD}

In this section we analyze the mixing time of Glauber dynamics with parallel updates using the path coupling
theorem. We will use the following distance function: for any $\sigma,\eta\in\Omega$,
\begin{eqnarray} \label{eq:distfun}
\Phi(\sigma,\eta) & = & \sum_{v}|\sigma_v-\eta_v|f(v) = \sum_{v\in\sigma\bigtriangleup\eta}f(v),
\end{eqnarray}
where $f(v)>0$ is a (weight) function of $v\in  V$ and recall that $\sigma\bigtriangleup\eta
=(\sigma\setminus\eta)\cup(\eta\setminus\sigma)$. Note that this distance function is a weighted Hamming
distance function and satisfies all the properties of a metric.

Consider the following \emph{coupling} $(\sigma(t),\eta(t))$: in every time slot both chains select the same
update set and use the same coin toss for every vertex in the update set if that vertex can be added to both
configurations.

Let $E[\Delta\Phi(\sigma(t),\eta(t))]$ be the (conditional) expected change of the the distance between the
states of the two Markov chains $\{\sigma(t)\}$ and $\{\eta(t)\}$ after one slot:
\begin{eqnarray*}
E[\Delta\Phi(\sigma(t),\eta(t))] = E[\Phi(\sigma(t+1), \eta(t+1)|\sigma(t), \eta(t)] - \Phi(\sigma(t),
\eta(t)).
\end{eqnarray*}
For any $\mathbf m\in \mathcal{I}$, let
$$E[\Delta^{\mathbf m}\Phi(\sigma(t),\eta(t))]=E[\Delta\Phi(\sigma(t),\eta(t))|\mathbf m \mbox{ is the update
set}].$$

\begin{lemma}\label{lem:independent updates}
Let $\tilde{\mathbf m}=(y_1,\ldots,y_{|\mathbf m|})$ be any ordering of $\mathbf m$. For any
$\sigma(t),\eta(t)\in\Omega$,
$$E[\Delta^{\mathbf m}\Phi(\sigma(t),\eta(t))]=\sum_{k=1}^{|\mathbf m|} E[\Delta^{y_k}\Phi(\sigma(t),\eta(t))].$$
\end{lemma}
\begin{proof}
Note that the value of $\Phi(\sigma,\eta)$ is completely determined by the set $\sigma\bigtriangleup\eta$,
which in turn depends only on $\sigma$ and $\eta$. Hence, it suffices to show that we will obtain the same
sets $\sigma(t+1)$ and $\eta(t+1)$ by updating all the vertices of $\mathbf m$ simultaneously and by updating
them in any sequential order.

The moves trying to remove vertices from a configuration will be successful in all cases. The outcome of a
move trying to add a vertex $y$ to a configuration $\omega$ is successful if and only if $\mathcal
N_y\cap\omega=\emptyset$. But $\mathbf m$ is an independent set, so no neighbor of a vertex $y\in\mathbf m$
is in $\mathbf m$. Then the states of the neighbors of $y$ are unchanged after updating any subset of
vertices of $\mathbf m$. Hence, $\forall k\in\{1,\ldots,|\mathbf m|\}$, the outcome of a move trying to add
$y_k$ to $\omega$ will be the same if we update the vertices of $\mathbf m$ sequentially and if we update all
the vertices of $\mathbf m$ simultaneously. This is in particular true for the configurations $\sigma(t)$ and
$\eta(t)$, so we conclude that the two update procedures will yield the same sets $\sigma(t+1)$ and
$\eta(t+1)$.
\end{proof}

We say that $\sigma,\eta\in\Omega$ are \emph{adjacent} and we write $\sigma\sim\eta$ if there exists $v\in V$
such that $\sigma$ and $\eta$ differ only at $v$. Let
$$S=\Big\{(\sigma, \eta):\sigma, \eta \in \Omega \mbox{ and } \sigma\sim\eta\Big\}.$$

Note that under the distance function defined in (\ref{eq:distfun}), for all $\sigma, \eta \in \Omega$, we
can find a path $\sigma=\tau_0, \tau_1, \ldots, \tau_{|\sigma \triangle \eta|}=\eta$ between $\sigma$ and
$\eta$ such that $(\tau_l,\tau_{l+1})\in S$ for $0\leq l <r$ and $\Phi(\sigma, \eta) =
\sum_{l=0}^{r-1}\Phi(\tau_l, \tau_{l+1}).$

Now consider a pair of adjacent configurations $\sigma(t)$ and $\eta(t)$ that differ only at $v$. Without
loss of generality, suppose $\sigma_v(t)=0$ and $\eta_v(t)=1.$ This means that, $\eta_w(t)=0$ for all $w\in
\mathcal{N}_v$. Since, $\sigma(t)$ and $\eta(t)$  differ only at $v$, this also means that $\sigma_w(t)=0$
for all $w\in \mathcal{N}_v$.

\begin{lemma}\label{lemma:dist_multiset}
\begin{eqnarray} \label{eq:dist_dif}
E[\Delta\Phi(\sigma(t),\eta(t))] \leq -q_v f(v)+\sum_{w\in\mathcal N_v}\frac{q_w \lambda_w}{1+\lambda_w}f(w).
\end{eqnarray}
\end{lemma}

\begin{proof}
Using Lemma~\ref{lem:independent updates}, we have
\begin{eqnarray*}
E[\Delta\Phi(\sigma(t),\eta(t))] & = &  E_{\mathbf m}\Big[E[\Delta^{\mathbf m}\Phi(\sigma(t),\eta(t))] \Big]\\
& = & \sum_{\mathbf m}q_{\mathbf m}E[\Delta^{\mathbf
m}\Phi(\sigma(t),\eta(t))]\\
& = & \sum_{\mathbf m}q_{\mathbf m}\sum_{y\in\mathbf m}E[\Delta^y\Phi(\sigma(t),\eta(t))]\\
& = & \sum_{y\in V}q_y E[\Delta^y\Phi(\sigma(t),\eta(t))].
\end{eqnarray*}

Note that only updates on vertices $v$ and $w \in \mathcal{N}_v$ can affect the value of
$E[\Delta\Phi(\sigma(t),\eta(t))]$. In particular, if $v$ is selected for update and since we use the same
coin toss for both Markov chains, $\sigma(t+1)=\eta(t+1)$. Thus $E[\Delta^v\Phi(\sigma(t),\eta(t))]=-f(v)$.

If $w\in \mathcal{N}_v$ is selected for update, under configuration $\eta(t),$ $w$ can only take value $0$
because $w$ has a neighbor (i.e., $v$) belongs to $\eta(t)$. While under configuration $\sigma(t),$ there are
two cases: \\
1) if $w$ has a neighbor in $\sigma(t)$, then $w$ can only take value $0$; \\
2) if $w$ has no neighbors in $\sigma(t)$, $w$ can take value $1$ with probability
$\frac{\lambda_w}{1+\lambda_w}$ and value $0$ otherwise.

Hence for $w\in \mathcal{N}_v$,
$$E[\Delta^w\Phi(\sigma(t),\eta(t))] \leq \frac{\lambda_w}{1+\lambda_w}f(w).$$ Summing up all contributions we
have (\ref{eq:dist_dif}).
\end{proof}

Now we are ready to present the main result of this paper.
\begin{theorem} \label{thm:main}
For any positive function $f(v)$ of $v\in V$, let $m=\min_{v\in V}f(v)$, $M=\max_{v\in V}f(v)$, and
$\xi=\frac{M}{m}$. If
\begin{eqnarray} \label{eq:theta}
\theta \triangleq \min_{v\in V}\left\{q_v f(v)-\sum_{w\in\mathcal N_v}\frac{q_w\lambda_w
}{1+\lambda_w}f(w)\right\} &
> & 0,
\end{eqnarray}
then the mixing time of the parallel Glauber dynamics is bounded by:
\begin{eqnarray} \label{eq:mixingPGB}
T_{\textit{mix}}(\epsilon) & \leq & \frac{M}{\theta}\log(\epsilon^{-1} n\xi).
\end{eqnarray}
\end{theorem}

\begin{proof}

For any pair of adjacent configurations $(\sigma(t), \eta(t))\in S$ that differ at some vertex $v\in V$, from
(\ref{eq:dist_dif}) and (\ref{eq:theta}) we have:
\begin{eqnarray*}
E[\Delta\Phi(\sigma(t),\eta(t))]  \leq  -\theta \leq -\frac{\theta}{M}\Phi(\sigma(t),\eta(t)),
\end{eqnarray*}
where we use the fact that
$$\Phi(\sigma(t),\eta(t)) = f(v) \leq M.$$
Therefore,
$$E[\Phi(\sigma(t+1),\eta(t+1))|\sigma(t),\eta(t)]\leq\big(1-\frac{\theta}{M}\big)\Phi(\sigma(t),\eta(t)).$$
Then, by applying the path coupling theorem where we let $\beta= 1-\frac{\theta}{M}$ and $D=n\xi$, we prove
the bound in (\ref{eq:mixingPGB}).
\end{proof}

\subsection{Conditions for Fast Mixing of Parallel Glauber Dynamics}

We can now specify the (weight) function $f$ to obtain different conditions on the fugacities $\lambda_v$'s
for fast mixing. We will show three such examples.

\begin{corollary}
Let $m=\min_{v\in V}\frac{1+\lambda_v}{q_v}$, $M=\max_{v\in V}\frac{1+\lambda_v}{q_v},$ and
$\xi=\frac{M}{m}.$ If
\begin{eqnarray}
\theta \triangleq \min_{v\in V}\left\{1+\lambda_v-\sum_{w\in\mathcal N_v}\lambda_w\right\} &
> & 0,
\end{eqnarray}
then we have
\begin{eqnarray}
T_{\textit{mix}}(\epsilon) & \leq & \frac{M}{\theta}\log(\epsilon^{-1} n\xi).
\end{eqnarray}
\end{corollary}
\begin{proof}
Choose $f(v)=\frac{1+\lambda_v}{q_v}$, $\forall v\in V$.
\end{proof}

\begin{corollary}
Let $q_{min}=\min_{v\in V}q_v$, $q_{max}=\max_{v\in V}q_v,$ and $\xi=\frac{q_{max}}{q_{min}}.$ If
\begin{eqnarray}
b \triangleq \max_{v\in V}\sum_{w\in\mathcal N_v}\frac{\lambda_w}{1+\lambda_w} &<& 1,
\end{eqnarray}
then we have
\begin{eqnarray} \label{eq:mixingPGB2}
T_{\textit{mix}}(\epsilon) & \leq & \frac{\log\big(\epsilon^{-1}n\xi\big)}{q_{min}(1-b)}.
\end{eqnarray}
\end{corollary}
\begin{proof}
Choose $f(v)=\frac{1}{q_v}$, $\forall v\in V$.
\end{proof}

\begin{remark}
If $q_{v}>c$ for some constant $c>0$ which is independent of the size of the network $n$, then the mixing
time is $O(\log n)$. In particular, for a bounded-degree graph G where $\delta$ and $\Delta$ are the minimum
and maximum vertex degrees, if $q_v=\frac{1}{d_v+1}$ where $d_v$ is the degree of $v$, then we have
$\frac{1}{\Delta+1}\leq q_v \leq \frac{1}{\delta+1}$ and
$$T_{\textit{mix}}(\epsilon)  \leq  \frac{\Delta+1}{1-b}\log\Big(\frac{\Delta+1}{\delta+1}\epsilon^{-1} n\Big).$$
\end{remark}

\begin{corollary}
If $\lambda_v<\frac{1}{d_v-1}$ for all $v\in V$, then
$T_{\textit{mix}}(\epsilon)\leq\frac{M}{\theta}\log(\epsilon^{-1} n\xi)$, for some constants $M$, $\xi$ and
$\theta>0$.
\end{corollary}

\begin{proof}
The proof is very similar to that of Theorem~\ref{thm:main}. We can choose $f(v)=\frac{d_v}{q_v}$ and let
$M=\max_{v\in V}\frac{d_v}{q_v}$ and
$$\xi=\frac{\max_{v\in V}\frac{d_v}{q_v}}{\min_{v\in V}\frac{d_v}{q_v}}.$$ The parallel Glauber dynamics will
have fast mixing if
$$\theta=\min_{v\in V}\left\{d_y-\sum_{w\in\mathcal N_y}\frac{\lambda_w}{1+\lambda_w}d_w\right\}>0.$$
To achieve that we need, for all $v\in V$,
$$d_v-\sum_{w\in\mathcal N_v}\frac{\lambda_w}{1+\lambda_w}d_w>0.$$
It is sufficient that $\frac{\lambda_v}{1+\lambda_v}d_v<1$, which is equivalent to
$\lambda_v<\frac{1}{d_v-1}$.
\end{proof}

Note that the condition $\lambda_v<\frac{1}{d_v-1}$ for all $v\in V$ might be very different from
$b=\max_{v\in V}\sum_{w\in\mathcal N_v}\frac{\lambda_w}{1+\lambda_w}<1$ (e.g., in a star network).

\section{Mixing Time of Single-Site Glauber Dynamics With Heterogenous Fugacities}  \label{sec:mixingHF}

Since the single-site Glauber dynamics is a special case of the parallel Glauber dynamics (in which the
update set always consists of only one vertex), the general results derived in the previous section also
apply to the single-site Glauber dynamics. The motivation of this section is to derive a larger region on the
fugacities under which the single-site Glauber dynamics is fast mixing, where we use a similar path coupling
technique as in \cite{vig01}.

We redefine the state space to be the set of all configurations $\Lambda^n$. For a vertex $v\in V$ and a
configuration $\sigma\in\Lambda^n$, we define the set of \emph{blocked} neighbors of $v$ with respect to
$\sigma$ as
$$B_\sigma(v)=\Big\{w\in\mathcal N_v: w \in \sigma \mbox{ or } \mathcal N_w\cap\sigma\neq\emptyset\Big\}$$
and of \emph{unblocked} neighbors of $v$ as
$$\bar{B}_{\sigma} (v)=\mathcal N_v\setminus B_\sigma(v).$$
Note that the unblocked neighbors of $v$ are the neighbors of $v$ that can be added to the configuration in
the next move. If $v\notin\sigma$, we write $\sigma^v$ to denote the configuration that differs from $\sigma$
only at $v$.

We say that $\sigma,\eta\in\Lambda^n$ are \emph{adjacent} and we write $\sigma\sim\eta$ if there exists $v\in
V$ such that $\sigma$ and $\eta$ differ only at $v$. Let
$$S=\Big\{(\sigma, \eta):\sigma\sim\eta\Big\}.$$
Then, for all $\sigma,\eta\in\Lambda^n$, we can define a path in $\Lambda^n$ between $\sigma$ and $\eta$ as a
sequence $(\tau_0,\ldots,\tau_r)\subseteq\Lambda^n$ such that for $0\leq i < r,\tau_i\sim\tau_{i+1}$ and
$\tau_0=\sigma,\tau_r=\eta$. Let $\mathcal P(\sigma,\eta)$ be the set of all paths in $\Lambda^n$ between
$\sigma$ and $\eta$.

We will use the following distance function on $\Lambda^n \times \Lambda^n$:
$\forall\sigma,\eta\in\Lambda^n$, let
$$\Phi(\sigma,\eta)=\mbox{min}_{(\tau_0,\ldots,\tau_r)\in\mathcal P(\sigma,\eta)}\sum_{i=0}^{r-1}l(\tau_i,\tau_{i+1}),$$
where
\begin{eqnarray} \label{eq:linklength}
l(\sigma,\sigma^v)=1+\frac{1}{2}\sum_{w\in\bar{B}_{\sigma}(v)}\lambda_w
\end{eqnarray}
can be viewed as the \emph{length} of edge $(\sigma, \sigma^v)$ according to the adjacent relationship
defined by $S$. $\Phi$ is clearly symmetric, non-negative, zero only when the configurations are identical,
and satisfies the triangle inequality (because $\mathcal P(\sigma,\eta)\subseteq\mathcal
P(\sigma,\mu)\times\mathcal P(\mu,\eta)$). Hence, it is indeed a metric on $\Lambda^n \times \Lambda^n$.

For all $v, u\in V$, we let $ \mathcal{T}(v,u)=\mathcal N_v\cap\mathcal N_u$ be the set of vertices that form
triangles with $v$ and $u$. To each vertex $v\in V$, we associate a fugacity $\lambda_v$.

Consider a pair of adjacent configurations $\sigma \sim \sigma^v$ (where $v\notin\sigma$). Suppose
$\sigma(t)=\sigma,$ $\eta(t)=\sigma^v$. The \emph{coupling} is simply that each configuration attempts the
same move at every time slot. More precisely, both Markov chains select the same vertex to update, and use
the same coin toss if the vertex can be added to the configurations. For convenience, we use the following
notations from \cite{vig01}. Let
$$E[\Delta\Phi]=E[\Phi(\sigma(t+1), \eta(t+1)|\sigma(t), \eta(t)] - \Phi(\sigma(t), \eta(t)).$$
This can be further calculated via the analysis of individual moves. Let
\begin{eqnarray*}
E[\Delta^{+y}\Phi]= E[\Delta \Phi | \mbox{both chains attempt to add } y \mbox{ at time
} t], \\
E[\Delta^{-y}\Phi] = E[\Delta \Phi | \mbox{both chains attempt to remove } y \mbox{ at time } t],
\end{eqnarray*}
and denote the total effect of all moves on $y$ by
\begin{eqnarray*}
E[\Delta^{y}\Phi] & = & \frac{\lambda_y}{1+\lambda_y}E[\Delta^{+y}\Phi] +
\frac{1}{1+\lambda_y}E[\Delta^{-y}\Phi].
\end{eqnarray*}

We provide the main result of this section in the following theorem, which says that the mixing time of
(single-site) Glauber dynamics with heterogenous fugacities is $O(n \log n)$ under certain conditions.
\begin{theorem} \label{theorem:SGB}
Let
\begin{equation}
 a =\max_{v\in V}\Big\{\sum_{w\in\mathcal N_v}\lambda_{w}\Big\}
\end{equation}
and $\gamma=\sum_{y\in V}(1+\lambda_y)$. If $ a <2$ and the probability of selecting vertex $y\in V$ to
update is $q_y = \frac{1+\lambda_y}{\gamma}$, then the mixing time of the Glauber dynamics is bounded by
\begin{eqnarray}
T_{\textit{mix}}(\epsilon) & \leq & \frac{\gamma}{1-\frac{ a }{2}}\log\Big(\epsilon^{-1} n(1+\frac{ a
}{2})\Big).
\end{eqnarray}
\end{theorem}

We need the following lemma to prove Theorem~\ref{theorem:SGB}, and its proof is given in the Appendix.

\begin{lemma}\label{biglemma}
\begin{equation}
2\sum_{y\in V}(1+\lambda_y)E[\Delta^y\Phi]\leq-2+\sum_{w\in\mathcal N_v}\lambda_w.
\end{equation}
\end{lemma}

\begin{proof} (Theorem~\ref{theorem:SGB})
Suppose $ a <2$ and $q_y = \frac{1+\lambda_y}{\gamma}$, then using Lemma~\ref{biglemma} we have
\begin{eqnarray*}
E[\Delta\Phi] & = & \sum_{y\in V}q_y E[\Delta^y\Phi] \\
& =  &  \frac{\sum_{y\in V}(1+\lambda_y)E[\Delta^y\Phi]}{\gamma} \\
& \leq & \frac{-1+\frac{ a }{2}}{\gamma} \\
& \leq & \frac{-1+\frac{ a }{2}}{\gamma}\Phi(\sigma(t),\eta(t)),
\end{eqnarray*}
where we use the fact that $\Phi(\sigma(t), \eta(t))\geq 1$ for any $\sigma(t) \neq \eta(t)$. Therefore,
$$E[\Phi(\sigma(t+1),\eta(t+1))|\sigma(t),\eta(t)]\leq(1+\frac{-1+\frac{ a }{2}}{\gamma})\Phi(\sigma(t),\eta(t)).$$

Since $\Phi(\sigma,\sigma^v)\geq1$, so for all $\sigma,\eta\in \Lambda^n$, $\sigma\neq\eta$, we have
$\Phi(\sigma,\eta)\geq1$. Also, $\Phi(\sigma,\sigma^v)\leq1+\frac{ a }{2}$, and thus $\Phi(\sigma,\eta)\leq
n(1+\frac{ a }{2})$. Then, apply the path coupling theorem where we let $\beta=1+\frac{-1+\frac{ a
}{2}}{\gamma}$ and $D=n(1+\frac{ a }{2})$:
\begin{eqnarray*}
T_{\textit{mix}}(\epsilon) & \leq & \frac{1}{1-(1+\frac{-1+\frac{ a }{2}}{\gamma})}\log\Big(
n(1+\frac{ a }{2})\epsilon^{-1}\Big) \\
& = & \frac{\gamma}{1-\frac{ a }{2}}\log\Big(
\epsilon^{-1}n(1+\frac{ a }{2})\Big) \\
& \leq & \frac{3n}{1-\frac{ a }{2}}\log( 2\epsilon^{-1}n).
\end{eqnarray*}
\end{proof}

\begin{example}
Consider a graph $G$ with $n$ vertices, labelled $1,2,\ldots,n.$ For each vertex $i$, let $\mathcal N_i=\{j:
j\neq i \mbox{ and } |j-i| \leq 3\}$ be the set of neighbors of $i$. For Glauber dynamics with homogeneous
fugacity $\lambda$, Vigoda's result \cite{vig01} says that it has mixing time $O(n\log n)$ when $\lambda  <
\frac{2}{\Delta-2}=0.5$, since the maximum vertex degree $\Delta=6$ in this graph. While for Glauber dynamics
with heterogenous fugacities, our result says that even the fugacity at some vertex exceeds $0.5$ (but is
less than $2$), the Glauber dynamics can still have $O(n\log n)$ mixing time, e.g., when $\lambda_i = 1$ for
vertices $i = 6l+1$, $l=0,1,\ldots,$ and $\lambda_j = 0.18$ for other vertices $j$.
\end{example}

\section{Conclusion}  \label{sec:conclusion}

In this paper we have analyzed the mixing time of a generalization of Glauber dynamics with parallel updates
and heterogeneous fugacities. By applying the path coupling theorem and by choosing appropriate distance
functions, we have obtained various conditions on the system parameters such as fugacities, vertex degrees
and update probabilities under which the parallel Glauber dynamics is fast mixing. In particular, we have
shown that the mixing time of the parallel Glauber dynamics grows as $O(\log n)$ in the number of vertices
$n$ for bounded-degree graphs when the fugacities satisfy certain conditions.

\appendix

\section{Proof of Lemma~\ref{biglemma}}

\begin{proof}
Note that when $a<2$, the edge length defined in (\ref{eq:linklength}) satisfies $1\leq l(\sigma,\sigma^v)
<2$. Hence the length of a path of two edges or more is at least $2$. This implies that the distance between
two adjacent configurations is simply the length of the direct edge between them, i.e.,
\begin{eqnarray*}
\Phi(\sigma, \sigma^v) = l(\sigma,\sigma^v) = 1+\frac{1}{2}\sum_{w\in\bar{B}_{\sigma}(v)}\lambda_w.
\end{eqnarray*}

Only updates on vertices within a distance of $2$ hops from $v$ can influence the value of
$\Phi(\sigma,\sigma^v)$. Hence,
\begin{eqnarray*}
& & 2\sum_{y\in V}(1+\lambda_y)E[\Delta^y\Phi] \\
& =& 2(1+\lambda_v)E[\Delta^v\Phi]+2\sum_{w\in\mathcal N_v}(1+\lambda_w)E[\Delta^w\Phi] +2\sum_{x\in \mathcal
N_w:w\in\mathcal N_v}(1+\lambda_x)E[\Delta^x\Phi].
\end{eqnarray*}

We will examine separately the terms involving $v$, $w$'s and $x$'s. It is important to notice that the moves
trying to remove a vertex from a configuration are always successful, whereas the moves trying to add a
vertex $y$ to a configuration $\sigma$ succeed only if $\mathcal N_y\cap\sigma=\emptyset$.
\\\\
$\mathbf {E[\Delta^v\Phi]}$:\\
Only if $\mathcal N_v\cap\sigma=\emptyset$, will the move trying to add $v$ to $\sigma$ be successful.
Therefore,
\begin{eqnarray*}
2(1+\lambda_v)E[\Delta^v\Phi] & = & \left\{ \begin{tabular}{ll} $-(1+\lambda_v)(2+\sum_{w\in\bar{B}_{\sigma}(v)}\lambda_w)$  & if $\mathcal N_v\cap\sigma=\emptyset$,\\
$-(2+\sum_{w\in\bar{B}_{\sigma}(v)}\lambda_w)$  & otherwise. \end{tabular} \right.
\end{eqnarray*}
\\\\
$\mathbf {E[\Delta^w\Phi]}$:\\
Trying to update a neighbor of $v$ can only increase the distance between the two configurations. Indeed,
either adding a neighbor of $v$ will fail, or it will increase the length of any path between the two
configurations by one, because it is then necessary to remove that neighbor from both configurations before
it may be possible to add $v$. Removing a neighbor of $v$ from the configuration will potentially unblock
some other neighbors of $v$ and thus increase the distance between the configurations.

More precisely: if $w\in B_\sigma(v)$, it is impossible to add $w$ to $\sigma$ and so $E[\Delta^{+w}\Phi]=0$;
otherwise, $E[\Delta^{+w}\Phi]=\Phi(\sigma_w,\sigma^v)-\Phi(\sigma,\sigma^v)$. By the triangle inequality,
$$\Phi(\sigma_w,\sigma^v)\leq\Phi(\sigma_w,\sigma_{w,v})+\Phi(\sigma_{w,v},\sigma^v).$$

We have $\Phi(\sigma_w,\sigma_{w,v})\leq\Phi(\sigma,\sigma^v)$ because
$\bar{B}_{\sigma_w}(v)\subseteq\bar{B}_{\sigma}(v)$. Also,
$2\Phi(\sigma_{w,v},\sigma^v)=2+\sum_{y\in\bar{B}_{\sigma^v}(w)}\lambda_y$ and
$\bar{B}_{\sigma^v}(w)=\bar{B}_{\sigma}(w)\setminus\{v\}\setminus \mathcal{T}(v,w)$. Then, by combining
everything, we get
\begin{eqnarray*}
2E[\Delta^{+w}\Phi]\leq\left\{ \begin{tabular}{ll} $0$ & if $w\in B_\sigma(v)$,\\
$2+\sum_{y\in\bar{B}_{\sigma}(w)\setminus\{v\}\setminus \mathcal{T}(v,w)}\lambda_y$ & otherwise.\end{tabular}
\right.
\end{eqnarray*}

Trying to remove $w$ will succeed and unblock some of $v$'s neighbors, as well as add $w$ to
$\bar{B}_{\sigma}(v)$ if $\mathcal N_w\cap\sigma=\emptyset$. A neighbor $w'$ of $v$ will be unblocked if and
only if $w$ is the only neighbor of $w'$ in $\sigma$. Thus,
\begin{eqnarray*}
2E[\Delta^{-w}\Phi]
&=&\left\{ \begin{tabular}{l} $0$  if $w\notin\sigma$,\\
$\delta(\mathcal N_w\cap\sigma=\emptyset)\lambda_w+\sum_{w'\in \mathcal{T}(v,w):\mathcal
N_{w'}\cap\sigma=\{w\}}\lambda_{w'}$ otherwise. \end{tabular} \right.
\end{eqnarray*}
where we use the indicator function $\delta(a)$ which is equal to $1$ if $a$ is true and $0$ otherwise.
\\\\
$\mathbf {E[\Delta^x\Phi]}$: \\
As in the case of the terms involving $w$'s, trying to remove a vertex $x$ from the configurations unblocks
some neighbors of $v$ and thus increases the distance between the configurations. On the other hand, adding a
vertex $x$ blocks some neighbors of $v$ and reduces the distance between the configurations. More precisely,
when removing $x$ from $\sigma$ and $\sigma^v$, a neighbor $w$ of $v$ will be unblocked if and only if $x$ is
the only neighbor of $w$ in $\sigma$. Thus,
\begin{eqnarray*}
2E[\Delta^{-x}\Phi]
&=&\left\{ \begin{tabular}{ll} $0$ & if $w\notin\sigma$,\\
$\sum_{w\in \mathcal{T}(v,x):\mathcal N_w\cap\sigma=\{x\}}\lambda_w$ & otherwise.\end{tabular} \right.
\end{eqnarray*}

Trying to add $x$ fails if $x\in\sigma$ or $\mathcal N_x\cap\sigma\neq\emptyset$. Otherwise, it succeeds and
blocks some neighbors of $v$. More precisely, a neighbor $w$ of $v$ is blocked if it is not in $\sigma$, it
was not already blocked in configuration $\sigma$ before the move, and it is a neighbor of $x$. Thus,
\begin{eqnarray*}
2E[\Delta^{+x}\Phi]=\left\{ \begin{tabular}{ll} $0$ & if $x\in\sigma$ or $\mathcal N_x\cap\sigma\neq\emptyset$,\\
$-\sum_{w\in\bar{B}_{\sigma}(v):x\in\mathcal N_w}\lambda_w$ & otherwise.\end{tabular} \right.
\end{eqnarray*}
We can rewrite this equation in the following form:
\begin{eqnarray*}
2E[\Delta^{+x}\Phi]=-\sum_{w\in\bar{B}_{\sigma}(v):x\in\bar{B}_{\sigma}(w)}\lambda_w.
\end{eqnarray*}\\

We are now ready to sum up the contributions from all the updates:
\begin{eqnarray}
& & 2\sum_{y\in V}(1+\lambda_y)E[\Delta^y\Phi] \nonumber \\
&\leq&-(1+\lambda_v\delta(\mathcal N_v\cap\sigma=\emptyset))\left(2+\sum_{w\in\bar{B}_{\sigma}(v)}\lambda_w\right) \nonumber \\
&+&\sum_{w\in\bar{B}_{\sigma}(v)}\lambda_w\left(2+\sum_{y\in\bar{B}_{\sigma}(w)\setminus\{v\}\setminus \mathcal{T}(v,w)}\lambda_y\right) \label{eqn8} \\
&+&\sum_{w\in\sigma}\left(\delta(\mathcal N_w\cap\sigma=\emptyset)\lambda_w+\sum_{w'\in \mathcal{T}(v,w):\mathcal N_{w'}\cap\sigma=\{w\}}\lambda_{w'}\right) \label{eqn9} \\
&+&\sum_{x\in\sigma}\left(\sum_{w\in \mathcal{T}(v,x):\mathcal N_w\cap\sigma=\{x\}}\lambda_w\right) \label{eqn10} \\
&-&\sum_{x:d(v,x)=2}\lambda_x\left(\sum_{w\in\bar{B}_{\sigma}(v):x\in\bar{B}_{\sigma}(w)}\lambda_w\right).
\label{eqn11}
\end{eqnarray}

We want to simplify the terms (\ref{eqn8}) + (\ref{eqn11}):
\begin{eqnarray*}
(\ref{eqn8}) + (\ref{eqn11}) & =
&\sum_{w\in\bar{B}_{\sigma}(v)}\lambda_w\left(2+\sum_{y\in\bar{B}_{\sigma}(w)\setminus\{v\}\setminus
\mathcal{T}(v,w)}\lambda_y-\sum_{x\in\bar{B}_{\sigma}(w)}\lambda_x\right).
\end{eqnarray*}
As we recall that the notation $x$ in $\{x\in\bar{B}_{\sigma}(w)\}$ refers to vertices that are at distance
$2$ from $v$ (as opposed to $y$ that refers to any vertex in $ V$), we get that
$$\{x\in\bar{B}_{\sigma}(w)\}=\bar{B}_{\sigma}(w)\setminus\{v\}\setminus \mathcal{T}(v,w),$$
that is the set $\bar{B}_{\sigma}(w)$ minus its elements at a distance $0$ or $1$ from $v$. Then, we can
simplify:
\begin{eqnarray*}
(\ref{eqn8}) + (\ref{eqn11})=2\sum_{w\in\bar{B}_{\sigma}(v)}\lambda_w.
\end{eqnarray*}

We now turn to (\ref{eqn9}) + (\ref{eqn10}):
\begin{eqnarray*}
& & (\ref{eqn9}) + (\ref{eqn10}) \\
&=&\sum_{w\in\sigma}\left(\delta(\mathcal N_w\cap\sigma=\emptyset)\lambda_w+\sum_{w'\in \mathcal{T}(v,w):\mathcal N_{w'}\cap\sigma=\{w\}}\lambda_{w'}\right) +\sum_{x\in\sigma}\left(\sum_{w\in \mathcal{T}(v,x):\mathcal N_w\cap\sigma=\{x\}}\lambda_w\right)\\
&=&\sum_{w\in\sigma:\mathcal N_w\cap\sigma=\emptyset}\lambda_w+\sum_{w\in B_\sigma(v):|\mathcal
N_w\cap\sigma|=1}\lambda_w \leq \sum_{w\in B_\sigma(v)}\lambda_w.
\end{eqnarray*}

Combining all these inequalities, we have
\begin{eqnarray*}
& & 2\sum_{y\in V}(1+\lambda_y)E[\Delta^y\Phi] \\
&\leq&-(1+\lambda_v\delta(\mathcal N_v\cap\sigma=\emptyset))\Big(2+\sum_{w\in\bar{B}_{\sigma}(v)}\lambda_w\Big)+2\sum_{w\in\bar{B}_{\sigma}(v)}\lambda_w+\sum_{w\in B_\sigma(v)\cup\sigma}\lambda_w\\
&=&-2+\sum_{w\in\mathcal N_v}\lambda_w-\lambda_v\delta(\mathcal N_v\cap\sigma=\emptyset)\Big(2+\sum_{w\in\bar{B}_{\sigma} (v)}\lambda_w\Big)\\
&\leq&-2+\sum_{w\in\mathcal N_v}\lambda_w.
\end{eqnarray*}
\end{proof}

\end{document}